\documentclass[12pt]{amsart}
\usepackage{psfrag}
\usepackage{color}
\usepackage{tikz}
\usetikzlibrary{matrix,arrows}
\usepackage{graphicx,graphics}
\usepackage{fullpage,amssymb,amsfonts,amsmath,amstext,amsthm,amscd,verbatim,enumerate}
\usepackage[T1]{fontenc}

\begin{document}

\newtheorem{theorem}{Theorem}[section]
\newtheorem{result}[theorem]{Result}
\newtheorem{fact}[theorem]{Fact}
\newtheorem{example}[theorem]{Example}
\newtheorem{conjecture}[theorem]{Conjecture}
\newtheorem{lemma}[theorem]{Lemma}
\newtheorem{proposition}[theorem]{Proposition}
\newtheorem{corollary}[theorem]{Corollary}
\newtheorem{facts}[theorem]{Facts}
\newtheorem{props}[theorem]{Properties}
\newtheorem*{thmA}{Theorem A}
\newtheorem{ex}[theorem]{Example}
\theoremstyle{definition}
\newtheorem{definition}[theorem]{Definition}
\newtheorem{remark}[theorem]{Remark}
\newtheorem*{defna}{Definition}
\newtheorem*{question}{Question}

\newcommand{\notes} {\noindent \textbf{Notes.  }}
\newcommand{\note} {\noindent \textbf{Note.  }}
\newcommand{\defn} {\noindent \textbf{Definition.  }}
\newcommand{\defns} {\noindent \textbf{Definitions.  }}
\newcommand{\x}{{\bf x}}
\newcommand{\z}{{\bf z}}
\newcommand{\B}{{\bf b}}
\newcommand{\V}{{\bf v}}
\newcommand{\T}{\mathbb{T}}
\newcommand{\Z}{\mathbb{Z}}
\newcommand{\Hp}{\mathbb{H}}
\newcommand{\D}{\mathbb{D}}
\newcommand{\R}{\mathbb{R}}
\newcommand{\N}{\mathbb{N}}
\renewcommand{\B}{\mathbb{B}}
\newcommand{\C}{\mathbb{C}}
\newcommand{\ft}{\widetilde{f}}
\newcommand{\dt}{{\mathrm{det }\;}}
 \newcommand{\adj}{{\mathrm{adj}\;}}
 \newcommand{\0}{{\bf O}}
 \newcommand{\av}{\arrowvert}
 \newcommand{\zbar}{\overline{z}}
 \newcommand{\xbar}{\overline{X}}
 \newcommand{\htt}{\widetilde{h}}
\newcommand{\ty}{\mathcal{T}}
\renewcommand\Re{\operatorname{Re}}
\renewcommand\Im{\operatorname{Im}}
\newcommand{\tr}{\operatorname{Tr}}

\newcommand{\ds}{\displaystyle}
\numberwithin{equation}{section}

\renewcommand{\theenumi}{(\roman{enumi})}
\renewcommand{\labelenumi}{\theenumi}

\title{The fast escaping set for quasiregular mappings}

\author{Walter Bergweiler}
\address{Mathematisches Seminar,
Christian--Albrechts--Universit\"at zu Kiel,
Lude\-wig--Meyn--Str.~4,
24098 Kiel,
Germany}
\email{bergweiler@math.uni-kiel.de}
\author{David Drasin}
\address{Department of Mathematics, Purdue University,
West Lafayette, IN 47907, USA}
\email{drasin@math.purdue.edu}
\author{Alastair Fletcher}
\address{Department of Mathematical Sciences, Northern Illinois University,
DeKalb, IL 60115-2888, USA}
\email{fletcher@math.niu.edu}

\begin{abstract}
The fast escaping set of a transcendental entire function is the set of all
points which tend to infinity under iteration as fast as compatible with
the growth of the function. We study the analogous set for quasiregular mappings in
higher dimensions and show, among other things, 
that various equivalent definitions of the fast
escaping set for transcendental entire functions in the plane also coincide
for quasiregular mappings.
We also exhibit a class of quasiregular mappings for which the fast
escaping set has the structure of a spider's web.
\end{abstract}

\maketitle

\section{Introduction}

Quasiregular mappings in $\R^m$ for $m\geq 2$ form a natural higher dimensional analogue of holomorphic mappings in the plane when $m\geq 3$, see section~\ref{quasi} for their definition and basic
properties. 
It is a natural question to ask to what extent the theory of complex dynamics carries over into higher dimensions; cf. the recent survey \cite{B2}.
The escaping set 
\[ I(f) = \{ x \in \R^m \colon f^n(x) \to \infty \}\]
plays an important role in the dynamics of entire functions ($m=2)$.
In \cite{BFLM} it is shown that the escaping set 
of a quasiregular mapping $f\colon\R^m \to \R^m$ of transcendental type is non-empty, and contains an unbounded component; more recently \cite{FN} studied $I(f)$ for quasiregular mappings of 
polynomial type, and there are more refined results when $f$ is uniformly quasiregular.

While $I(f)$ was introduced in ~\cite{E}, the fast escaping set $A(f)$ of a transcendental entire function 
first appeared later in~\cite{BH}, and since has been the subject of
much recent study, see for example Rippon and Stallard \cite{RS}. For
certain classes of entire functions ($m=2$), 
in particular for all
functions which grow slowly enough \cite{RS2012},
$A(f)$ 
has a topological structure called a spider's web. Among papers which present
classes of functions for which $A(f)$ has a spider's web
structure are \cite{MBP,O,RS}. This notion will also be investigated here.

According to \cite{RS}, there are three equivalent ways of defining $A(f)$ for entire functions. To extend notions of complex dynamics to higher dimensions, we fix one of these to define the fast escaping set for a quasiregular mapping. 

Let $E$ be a bounded set in $\R^m$.  
Its {\em topological hull} $T(E)$ is the union of $E$ and its bounded
complementary components;
informally, $T(E)$ is $E$ with the holes filled in. 
 We note that in~\cite{RS,RS2012} and other 
papers on complex dynamics the notation $\widetilde{E}$ has been used instead.
The established notation $T(E)$, which appears for example in \cite{Bell1967} 
or \cite{Blokh2009}, is advantageous when working with ``complicated'' sets~$E$.
The set $E$ is called {\em topologically convex} if $T(E)=E$.

It was shown in~\cite[Lemma 5.1 (ii)]{BFLM} that if 
 $f\colon\R^m \to \R^m$ is a quasiregular mapping of transcendental type,
then there exists $R_0>0$ such that if $R>R_0$,
there is a sequence $(r_n)_{n=1}^{\infty}$ with $r_n\to \infty$ such that
\begin{equation} 
\label{1a}
T(f^n(B(0,R)))\supset B(0,r_n).
\end{equation}
Here and in the following $B(x,r)$ is the open ball of radius 
$r$ about $x\in\R^m$. {{When $x=0$ we often write $B(0,r)$ as $B(r)$ or, when
the specific $r$ is clear, $B$.}}

\begin{definition}
Let $f\colon\R^m \to \R^m$ be a quasiregular mapping of transcendental type
and let $R> R_0$ where $R_0$ is chosen such that~\eqref{1a} holds.
Then 
\begin{equation} 
\label{feqr1}
A(f) = \{ x \in \R^m \colon \exists L \in \N\;  \forall n \in \N\colon  f^{n+L}(x) \notin T(f^n(B(0,R))) \}
\end{equation}
is called the {\em fast escaping set}. 
\end{definition}

We will see later (Proposition~\ref{properties}) that this definition does
not depend on~$R$.
Our first theorem 
extends results of~\cite{BH,RS1} to the quasiregular context.

\begin{theorem}
\label{mainthm}
Let $f\colon\R^m \to \R^m$ be a quasiregular mapping of transcendental type. Then $A(f)$ is non-empty and every component of $A(f)$ is unbounded.
\end{theorem}

In \cite{RS}, equivalent definitions of the
fast escaping set for transcendental entire functions $f\colon \C \to \C$ are presented in terms of the
maximum modulus
\[
M(r,f)=\max_{|z|=r}|f(z)|.
\]
These definitions are
\begin{equation}
\label{fe2}
A_1(f) = \{ z \in \C \colon \exists L \in \N\;\forall n\in\N\colon  |f^{n+L}(z)| > M(R,f^n)\},
\end{equation}
where $R> \min _{z \in J(f)}|z|$ and $J(f)$ is the Julia set, and
\begin{equation}
\label{fe3} 
A_2(f) = \{ z \in \C \colon \exists L \in \N\;\forall n\in\N\colon |f^{n+L}(z)| \geq M^n(R,f)\},
\end{equation} 
where $M^n(R,f)$ is the $n$th iterate of $M(R,f)$
with respect to the first variable (for example, $M^2(R,f) = M(M(R,f),f)$),
with $R$ 
so large that $M^n(R,f)\to\infty$ as $n\to\infty$. 

For entire functions in the plane, the analogue of \eqref{fe2} was the first definition used, whereas now \eqref{fe3} has become the standard definition for $A(f)$.
We next show that the generalizations of these two alternate definitions also coincide with our initial definition in the quasiregular case.

\begin{theorem}
\label{fastescagreethm}
Let $f\colon\R^m \to \R^m$ be a quasiregular mapping of transcendental type.
Then $A(f) = A_1(f) = A_2(f)$, where $A_1(f)$ and $A_2(f)$ are the natural
generalizations of \eqref{fe2} and \eqref{fe3} to quasiregular mappings,
with $R>R_0$ and $R_0$ chosen such that~\eqref{1a} holds.
\end{theorem}

The main tool in the proof of this theorem is the following result.
\begin{theorem}
\label{1x}
Let $f\colon \R^m\to\R^m$ be a quasiregular mapping of
transcendental type and let $0<\eta<1$. Then there exists  $R_1>0$ such that
\[
\overline{B(0,M^n(\eta R,f))} \subset T(f^n(B(0,R)))
\]
for all $R>R_1$ and all $n\in\N$.
\end{theorem}

Next, we show that for a certain class of quasiregular mappings,
including those constructed by Drasin and Sastry \cite{DS}, the 
fast escaping set $A(f)$ has a particular structure called a spider's web.

\begin{definition}
A set $E\subset \R^m$ is a \emph{spider's web} if $E$ is connected and there exists a sequence $(G_n)_{n\in\N}$ of bounded topologically convex
domains satisfying $G_n \subset G_{n+1}$, $\partial G_n \subset E$ 
for $n \in \N$ and such that $\bigcup_{n \in \N} G_n = \R^m$.
\end{definition}
This definition in principle allows $\R^m$ itself to be a spider's web, but since a
quasiregular mapping of transcendental type has infinitely many periodic points \cite{Siebert2006}, we cannot have $A(f) = \R^m$.

The quasiregular maps constructed by Drasin and Sastry
behave like power maps (see \cite[p.13]{Rickman}) in large annuli and hence 
for these maps $f$ the minimum modulus 
\[
m(r,f)=\min_{|x|=r}|f(x)|
\]
is large for most values of~$r$. In particular, these maps satisfy the
hypotheses of the next theorem.

\begin{theorem}
\label{dsspider}
Let $f\colon \R^m\to\R^m$ be a quasiregular mapping of
transcendental type.
Suppose there exist $\alpha>1$ and $\delta>0$ such that
for all large $r$ there exists $s\in [r,\alpha r]$ such that
$m(s,f)\geq \delta M(r,f)$.
Then $A(f)$ is a spider's web.
\end{theorem}

\begin{remark}
A different situation occurs for maps with a Picard exceptional
value, such as Zorich-type mappings considered in \cite{B1}, where
the hypothesis of Theorem \ref{dsspider} fail. 
Quite generally, mappings which are periodic or bounded on a path to infinity
cannot satisfy the hypothesis of Theorem~\ref{dsspider}.
For the maps studied in \cite{B1} the escaping set --  as well as the fast escaping
set -- forms hairs, in analogy with the exponential family in the plane. 
\end{remark}

Theorem \ref{fastescagreethm} shows that the various formulations of $A(f)$ agree for quasiregular mappings of transcendental type. 
One way to show this for transcendental entire functions in the plane has been to use Wiman-Valiron theory, see \cite{E,RS}. This method shows, in
 particular, that near most points where the maximum modulus is achieved, a transcendental entire function behaves like a power mapping and maps a neighbourhood of such a point onto a large annulus, a property
 which may be iterated.

\begin{question}
Is there an analagous annulus covering theorem for neighbourhoods of points where the maximum modulus is achieved for quasiregular mappings of transcendental type?
\end{question}
The weaker covering result given by Proposition~\ref{keyprop} below is
sufficient for our purposes.

Wiman-Valiron theory is based on the power series
expansion of an entire function, which has no analogue for quasiregular maps.
An alternative approach to Wiman-Valiron theory, more in the
spirit of Macintyre's theory of flat regions~\cite{Macintyre1938},
was developed in~\cite{BRS}. Here, as well as in
classical Wiman-Valiron theory, one of the key features is
the convexity of $\log M(r,f)$ in $\log r$. 
We construct a variation of the mapping of Drasin and Sastry \cite{DS}, and
generalize unpublished ideas of Dan Nicks, to show that this need not be the
case for quasiregular mappings in $\R^m$.

\begin{theorem}
\label{notconvex}
Let $\varepsilon >0$. There exists a quasiregular mapping $F\colon\R^m \to \R^m$ for which $\log M(r,F) / \log r$ is decreasing on a collection of intervals whose union has
lower logarithmic density at least $1-\varepsilon$.
\end{theorem}

Recall that the lower logarithmic density of a set $A \subset [1,\infty)$ is
\[ \liminf_{r\to \infty} \int_{A \cap [1,r]} \frac{dt}{t}.\]

\begin{remark}
On intervals where $\log M(r,f)/\log r$ is decreasing, $\log M(r,f)$ cannot be convex in $\log r$.
Note however that by \cite[Lemma 3.4]{B}, $\log M(r,f)/\log r \to \infty$ as $r \to \infty$ for quasiregular mappings of transcendental type.
\end{remark}

This paper is organized as follows. In section~\ref{prelim}, we recall some
basic material on quasiregular mappings and prove some topological lemmas
needed in the sequel.
Section~\ref{basic} contains some basic results on the fast escaping set
and then establishes Theorem \ref{mainthm} in section~\ref{proofmainthm}.
In section~\ref{proof1x},  we prove Theorems~\ref{fastescagreethm} and \ref{1x}, while
section~\ref{further} lists properties of $A(f)$ which extend directly to
the quasiregular setting from \cite{RS}. These results are then
used to prove Theorem~\ref{dsspider}.
In
section~\ref{spiderexample}, we recall the mappings
of Drasin and Sastry, and prove that they satisfy the hypotheses of
Theorem~\ref{dsspider}, hence giving example of mappings for which that the fast escaping set is a spider's web.
Finally, section~\ref{failureconvex} presents Theorem \ref{notconvex}.

The authors thank Dan Nicks for kind sharing of unpublished ideas in relation to Theorem~\ref{notconvex}, stimulating discussions and helpful comments.

\section{Preliminaries}
\label{prelim}
\subsection{Quasiregular maps}
\label{quasi}
A continuous mapping $f\colon U \rightarrow \R^{m}$ defined on a domain $U \subset \R^{m}$ is {\it quasiregular} if $f$ belongs to the Sobolev space $W^{1}_{m, loc}(U)$ and there exists $K \in [1, \infty)$ with 
\begin{equation}
\label{eq2.1}
\av f'(x) \av ^{m} \leq K J_{f}(x)
\end{equation}
almost everywhere in~$G$. Here $|f'(x)|=\sup_{|h|=1}|f'(x)|$ is
the norm of the derivative $f'(x)$ and 
$J_{f}(x)=\det f'(x)$ the Jacobian determinant of $f$ at $x \in G$. The smallest constant $K \geq 1$ for which (\ref{eq2.1}) holds is the 
\emph{outer dilatation} $K_{O}(f)$. When $f$ is quasiregular we also have
\begin{equation}
\label{eq2.2}
J_{f}(x) \leq K' \inf _{\av h \av =1} \av f'(x) h \av ^{m}
\end{equation}
almost everywhere in $U,$ for some $K' \in[1, \infty)$. The smallest constant $K' \geq 1$ for which (\ref{eq2.2}) holds is the 
\emph{inner dilatation} $K_{I}(f)$. The \emph{dilatation} $K(f)$ of $f$ is the 
maximum of $K_{O}(f)$ and $K_{I}(f)$, 
and we say that $f$ is $K$-quasiregular if $K(f)\leq K$.
Informally, a quasiregular mapping sends infinitesimal spheres to infinitesimal
ellipsoids with bounded eccentricity. Quasiregular mappings generalize to higher
dimensions the mapping properties of analytic and meromorphic functions in the plane;
see Rickman's monograph \cite{Rickman} for many more details.
In particular, quasiregular mappings are open and discrete.

Quasiregular mappings share some appropriately-modified value distribution properties with holomorphic functions in the plane. 
Rickman~\cite{Rickman80} proved the existence of a constant $q=q(m,K)$ such that if a $K$-quasiregular mapping $f\colon\R^m \to \R^m$ omits at least $q$ values in $\R^m$, then $f$ is constant. 
This number $q$ is called Rickman's constant, and this result becomes an extension of Picard's Theorem in the plane; for fixed $m\geq 3$, \cite{DP} shows that $q(m,K)\to\infty$ as $K\to \infty$, the case $m=3$ being due to Rickman~\cite{Rickman85}.
Miniowitz obtained an analogue of Montel's Theorem for quasiregular mappings with poles, i.e.\ quasiregular mappings $f\colon U\to \overline{\R^m}$
where $\overline{\R^m}=\R^m\cup\{\infty\}$.

\begin{theorem}[\cite{Miniowitz}, Theorem 5]
\label{minio}
Let $\mathcal{F}$ be a family of $K$-quasimeromorphic mappings in a domain $U\subset \R^m$,
$m \ge 2$. Let $q=q(m,K)$ be Rickman's constant. Suppose there exists a
positive number $\varepsilon$ such that
\begin{enumerate}[(i)]
\item each $f \in \mathcal{F}$ omits $q+1$ points $a_1(f), \ldots, a_{q+1}(f)$ in $\overline{\R^m}$,
\item $\chi(a_i(f), a_j(f)) \ge \varepsilon $, where $\chi$ is the spherical metric on $\overline{\R^m}$.
\end{enumerate}
Then $\mathcal{F}$ is a normal family.
\end{theorem}

A quasiregular mapping $f\colon\R^m \to \R^m$ is of \emph{polynomial type}
if $|f(x)| \to \infty$ as $|x| \to \infty$, whereas it is said to be of 
\emph{transcendental type} if this limit does not exist, so that $f$ has an essential singularity at infinity. This is in direct analogy with the dichotomy between polynomials and transcendental entire functions when $m=2$.
The following lemma was proved in~\cite[Lemma 3.3]{B}.
\begin{lemma} \label{MAr}
Let $f\colon \R^m\to\R^m$ be quasiregular mapping of transcendental type
and $A>1$. Then
\[
\lim_{r\to\infty}\frac{M(Ar,f)}{M(r,f)}=\infty.
\]
\end{lemma}

The composition of two quasiregular mappings is always quasiregular, but the dilatation typically increases. See \cite{B2} for an introduction to the iteration theory of quasiregular mappings,
as well as \cite{B2013,BN}, where a Fatou-Julia iteration theory for
quasiregular mappings is developed.

\subsection{Some topological lemmas}
For the following lemma we refer to \cite[p.~84]{Newman} 
and \cite[Lemmas 3.1 and 3.2]{BFLM}.

\begin{lemma}
\label{bflmlemma}
Let $E$ be a continuum in $\overline{\R^m}$ containing $\infty$. Then:
\begin{enumerate}
\item if $F$ is a component of $\R^m \cap E$, then $F$ is unbounded;
\item if $f\colon \R^m \to \R^m$ is a continuous open mapping, then the pre-image
\[ f^{-1}(E) = \{ x \in \R^m \colon  f(x) \in E \}\]
cannot have a bounded component.
\end{enumerate}
\end{lemma} 

\begin{proposition}
\label{twiddlesprop}
Let  $f\colon \R^m \to \R^m$ be a continuous open mapping and 
$U \subset \R^m$ a bounded open set.
Then $f(T(U)) \subset T(f(U))$ and $\partial T(f(U)) \subset f(\partial T(U))$.
\end{proposition}

\begin{proof}
We apply Lemma~\ref{bflmlemma} with $E=\overline{\R^m}\backslash T(f(U))$
and $F=\R^m\backslash T(f(U))$.
Thus $f^{-1}(E)=f^{-1}(F)$ has no bounded component 
(in fact, $f^{-1}(F)$ consists of a single unbounded component, but we 
do not need this fact).
Since $f(U)\subset T(f(U))$ we have $F\subset \R^m\backslash f(U)$
and hence 
\[
f^{-1}(F) \subset f^{-1} ( \R^m\backslash f(U) ) \subset \R^m\backslash U.
\]
Thus every component of $f^{-1}(F)$ is contained in a component of 
$\R^m\backslash U$. Since $f^{-1}(F)$  has no bounded components,
we deduce that $f^{-1}(F)$ is contained in the unique unbounded component
of $\R^m\backslash U$. This means that
$f^{-1}(F) \subset \R^m\backslash T(U)$. 
Now
\[
f^{-1}(F) =f^{-1} (\R^m\backslash T(f(U))) = \R^m\backslash f^{-1} (T(f(U))).
\]
This yields that $ T(U)\subset f^{-1} (T(f(U)))$
and hence $f(T(U))\subset T(f(U))$.

For the second part of the proposition, observe that $\partial f(U)
\subset f(\partial U)$ since $f$ is a continuous open mapping.
Indeed, let $w$ be in the
boundary of $f(U)$. Then $w$ is not in $f(U)$ since $f(U)$ is open,
but $w$ is the limit of points $w_k=f(u_k)$ with $u_k$ in $U$.
Without loss of generality $u_k$ tends to a point $u$. Then
$w=f(u)$ and since $w$ is not in $f(U)$, $u$ is not in $U$. Thus
$u$ is in the boundary of $U$. Hence
$\partial f(U) \subset f(\partial U)$ as claimed.

From this it follows that
\[ \partial T(f(U)) \subset \partial f(U) \subset f(\partial U).\]
If $A$ is a bounded component of $\R^m \backslash U$, then $A \subset T(U)$
and thus $f(A) \subset f(T(U)) \subset T(f(U))$ by the first part of the proposition. Thus all components of the boundary of $U$ other than $\partial T(U)$ are mapped into $T(f(U))$,
that is, 
\[ f(\partial U\backslash \partial T(U)) \subset T(f(U)).\]
Together with $\partial T(f(U)) \subset f(\partial U)$ this
yields $\partial T(f(U)) \subset f(\partial T(U))$.
\end{proof}
Since non-constant quasiregular mappings are open and
discrete, these results apply in particular to quasiregular mappings.

\section{Basic properties of $A(f)$}
\label{basic}

The following proofs mimic those in \cite{RS1,RS} for entire functions.

\begin{proposition}
\label{properties}
Let $f\colon \R^m \to \R^m$ be a quasiregular mapping of transcendental type.
Then:
\begin{enumerate}
\item $A(f)$ is independent of $R$ as long as $R>R_0$, where $R_0$ satisfies \eqref{1a};
\item for any $p \in \N$, $A(f^p) = A(f)$;
\item $A(f)$ is completely invariant under $f$, i.e.\ if $x\in A(f)$ then $f(x) \in A(f)$ and 
vice versa.
\end{enumerate}
\end{proposition}

\begin{proof}

For property (i), suppose $R'>R>R_0$.
Of course $T(f^n(B(0,R))) \subset T(f^n(B(0,R')))$ for each $n \in\N$,
and so we just need that $T(f^n(B(0,R)))$ covers $B(0,R')$ for
some $n \in \N$. However, this follows immediately from \eqref{1a} by choosing $n$ large enough.
Since $A(f)$ does not depend on $R$, we may and do assume
that~\eqref{feqr1} is satisfied. 

For property (ii), if $x \in A(f)$, there exists $L \in \N$ such that
$f^{n+L}(x) \notin T(f^n(B))$ for all $n \in \N$, where $B=B(0,R)$.
According to  \eqref{1a}, $T(f^k(B)) \supset B$ for all $k \in \N$, so
Proposition~\ref{twiddlesprop} yields $T(f^{n+k}(B)) \supset T(f^n(B))$ 
for $n \in \N$. Hence $f^{pn+L+k}(x) \notin T(f^{pn}(B))$ for $n \in \N$.
Choosing $k$ such that $p$ divides $L+k$, we conclude $x \in A(f^p)$.
Conversely, if $x \in A(f^p)$, then there exists $\ell \in \N$ with
$(f^p)^{n+\ell}(x) \notin T((f^p)^n(B))$ for $n \in \N$. Using
Proposition~\ref{twiddlesprop}, $f^{pn+p\ell-k}(x) \notin T(f^{pn-k}(B))$
for $n \in \N$ and $k=1,2,\ldots,p-1$. Hence $x \in A(f)$.

Finally, for property (iii), if $x\in A(f)$ and $y\in f^{-1}(x)$, then
$y$ satisfies (\ref{feqr1}) with $L$ replaced by $L+1$.
Also, if $z=f(x)$ then $f^{n+L}(z) = f^{n+L+1}(x) \notin T(f^{n+1}(B))$. Then
by arguments similar to those for property (ii),
Proposition~\ref{twiddlesprop} implies that $f^{n+L}(x) \notin T(f^n(B))$.
\end{proof}

\section{Proof of Theorem \ref{mainthm}}
\label{proofmainthm}
This was essentially proved, but not stated, in \cite{BFLM}.
By Proposition~\ref{properties} (i),
we may choose $R$ such that~\eqref{1a} holds.
We first show that $A(f)\neq \emptyset$. 
For $n \in \N$ let $\gamma_n = \partial T(f^n(B))$, 
where $B=B(0,R)$.
Then $\gamma _{n+1} \subset f(\gamma_n)$ by the second part of Proposition \ref{twiddlesprop}.
By \cite[Lemma 5.2]{BFLM} there is a point $x_0 \in \R^m$ such that
\[ f^n(x_0) \in \gamma_n,\]
for each $n \in \N$. In particular,  $x_0 \in I(f)$. 
However, since each $T(f^n(B))$ is open, we have
\[ x_0 \notin T(f^n(B))\]
for $n \in \N$, and so $x_0 \in A(f)$;
thus $A(f)\neq\emptyset$.

To prove that the components of $A(f)$ are unbounded, 
write $B_n = f^n(B)$ and $E_n =\R^m\backslash T(B_n)$. Suppose that $x_0 \in \R^m$ satisfies
\begin{equation}\label{notinEn}
f^n(x_0) \in E_n 
\end{equation}
for all $n \in \N$. Let $L_n$ be the component of $f^{-n}(E_n)$ which
contains $x_0$. Then $L_n$ is obviously closed, and is unbounded by Lemma \ref{bflmlemma} (ii). Further,
for $n\in \N$,
\[ L_{n+1} \subset L_n.\]
To see this, note from Proposition \ref{twiddlesprop} that $f^{n+1}(x) \in E_{n+1}$ implies that $f^n(x) \in E_n$ so $L_{n+1} \subset f^{-n}(E_n)$.
Therefore, by \cite[Theorem 5.3, p. 81]{Newman}, 
\[ K := \bigcap_{n \in \N} \left ( L_n \cup \{ \infty \} \right )\]
is a closed connected subset of $\overline{\R^m}$
which contains $x_0$ and $\infty$. Now let $K_0$ be the component of $K \backslash \{ \infty \}$ containing  $x_0$. Then $K_0$ is closed in $\R^m$, and unbounded by Lemma \ref{bflmlemma} (i). We claim that $K_0 \subset A(f)$. To see this, observe that if $x \in K_0$, then $f^n(x) \in E_n$ for $n \in \N$ and so
\[ f^n(x) \notin T(B_n) = T( f^n(B(0,R))),\]
for $n \in \N$. Hence $x \in A(f)$.
We have shown that if $x_0$ satisfies~\eqref{notinEn}, then
$x_0$ is contained in an unbounded component $K_0$ of $A(f)$.

Next, suppose that $x \in A(f)$. Then by (\ref{feqr1}), there exists $L \in \N$ such that
\[ f^{n+L}(x) \in E_n,\]
for $n \in \N$ and so $y= f^L(x)$ satisfies
\[ f^n(y) \in E_n,\]
for $n \in \N$. By the argument above, $y$ lies in an unbounded closed connected subset $K'$ of $A(f)$. 
Lemma \ref{bflmlemma} (ii) implies that if $K''$ is the component of $f^{-L}(K')$ containing~$x$, then $K''$ is closed and unbounded.
Since $A(f)$ is completely invariant by Proposition \ref{properties} (iii), it follows that $K'' \subset A(f)$.
This completes the proof of Theorem \ref{mainthm}.

\section{Proof of Theorems \ref{fastescagreethm} and \ref{1x}}
\label{proof1x}
We begin with the following proposition
where 
\[ A(R_1,R_2) = \{ y\in \R^m \colon R_1 < |y| < R_2 \} \]
is the annulus centred at $0$ with radii $R_1,R_2$.
\begin{proposition}
\label{keyprop}
Let $f\colon \R^m\to\R^m$ be a quasiregular mapping of transcendental type
and let $\alpha, \beta>1$.
Then, for all large enough $r$, there exists $R>M(r,f)$ such that
\[ f(A(r,\alpha r)) \supset A(R,\beta R).\]
\end{proposition}

\begin{proof}
Choose $r$ large and a point $a_r \in \R^m$ with $|a_r|=(1+\alpha)r/2$ and 
\[ M \left ( (1+\alpha)r/2, f \right) = |f(a_r)|.\] Define $g_r\colon B(0,1) \to \R^m$ by
\[ g_r(x) = \frac{ f(a_r + (\alpha-1)rx/2 )}{|f(a_r)|}.\]
For a fixed $\rho \in (0,1)$ put $b_r = -2\rho a_r/((1+\alpha)r)$. Then $|b_r| = \rho$ and 
\[
|g_r(b_r)| = \frac{ |f(a_r - 2\rho a_r/(1+\alpha ) ) |}{|f(a_r)|} \leq \frac{ M ( (1-2\rho/(1+\alpha) )(1+\alpha)r/2 ,f)}{M( (1+\alpha)r/2, f)}.
\]
Combining this with Lemma \ref{MAr}, we see that $|g_r(b_r)| \to 0$ as $r \to \infty$ and so 
\[ \min_{|x|=\rho} |g_r(x)| \to 0\]
as $r \to \infty$ for all $\rho \in (0,1)$. As $|g_r(0)|\equiv1$, this implies that the family of $K$-quasiregular mappings $\{ g_r \colon  r>0\}$ is not normal. 
In fact, for any sequence $(r_k)$ tending to $\infty$, the family $\{ g_{r_k} \colon  k \in \N\}$ is not normal.

Let $q=q(m,K)$ be Rickman's constant and  $\beta>1$. It follows from
Miniowitz's version of Montel's Theorem, i.e.\ Theorem \ref{minio}, that if
$r$ is large enough then there exists $p=p_r \in \{0,1,\ldots, q\}$ such
that $g_r(B(0,1)) \supset A(\beta^{2p}, \beta^{2p+1})$. This implies that
\[ f( A(r,\alpha r) ) \supset f( B(a_r, (\alpha -1)r/2 ) ) 
\supset A( \beta^{2p} |f(a_r)|, \beta^{2p+1} |f(a_r)| ). \]
Since $\beta^{2p}|f(a_r)| \geq |f(a_r)| = M((1+\alpha) r/2, f) > M(r,f)$, the conclusion follows.
\end{proof}
\begin{proof}[Proof of Theorem \ref{1x}]
Put $\alpha =1/\eta$, let $\beta \geq\alpha$ and choose $r$ large enough
so we may apply Proposition \ref{keyprop}. Then
\[ \overline{B(0,M(r,f) )} \subset B(0,\beta M(r,f) ) \subset T(f(B(0,\alpha r) )).\]
Further, by Proposition \ref{twiddlesprop}, Proposition \ref{keyprop} and the fact that if $U\subset V$ then $T(U) \subset T(V)$, we have
\begin{align*}
\overline{B(0,M^2 (r,f) )} &= \overline{ B(0, M(M(r,f),f) )} \\
&\subset T(f( B(0,\alpha M(r,f) ) )) \\
& \subset T( f( B(0, \beta M(r,f)))) \\
&\subset T( f^2(B(0,\alpha r ) )) .
\end{align*}
Continuing by induction and replacing $r$ with $\eta R=R/\alpha$ 
yields the conclusion.
\end{proof}

\begin{proof}[Proof of Theorem \ref{fastescagreethm}]
As in the proof of Proposition \ref{properties}, we can
show that the definitions of $A_1(f)$ and $A_2(f)$ also do not 
depend on $R$ as long as $R>R_0$.
Given $\eta\in (0,1)$, 
Theorem~\ref{1x} implies that
\begin{align*}
\overline{B(0, M^n(\eta R,f)) } & \subset T(f^n(B(0,R))) \\
&\subset \overline{ B(0, M(R,f^n))} \\
&\subset \overline{ B(0, M^n(R,f)) } \\
&\subset B(0, M^{n+1}(R,f))
\end{align*}
for large $R$,
from which the conclusion easily follows.
\end{proof}

\section{Further properties of $A(f)$}
\label{further}
Many results on the structure of the fast escaping set for entire
functions from Rippon and Stallard's paper \cite{RS} hold in this context. In this section, we state these results and refer to \cite{RS} for the proofs, where they go through almost word for word.

\begin{definition}
Let $R>R_0$ 
with $R_0$ as in \eqref{1a} 
and let $L \in \Z$. Then
\[ A_R(f) = \{ x \in \R^m \colon  |f^n(x)| \geq M^{n}(R,f)
\text{ for all } n \in \N\}, \]
is the {\em fast escaping set with respect to $R$}, and its $L$'th {\em level} is
\[ A_R^L(f) = \{ x \in \R^m \colon  |f^n(x)| \geq M^{n+L}(R,f)
\text{ for all } n \in \N \text{ with } n \geq -L \}. \]
\end{definition}
Note that $A_R(f) = A_R^0(f)$ and that $A_R^L(f)$ is a closed set. 
Since $M^{n+1}(R,f) > M^n(R,f)$, we have
\[ A_R^L(f) \subset A_R^{L-1}(f),\]
for $L \in \Z$. Therefore  $A(f)$ as defined in (\ref{feqr1}) is an increasing union of closed sets
\[ A(f) = \bigcup_{L \in \N} A_R^{-L}(f).\]
We also note 
\[ A_R^L(f) \subset \{ x \in \R^m \colon  |x| \geq M^L(R,f) \},\]
for $L \geq 0$ and that
\[ f(A_R^L(f)) \subset A_R^{L+1}(f) \subset A_R^L(f),\]
for $L \in \Z$.

\begin{proposition}
\label{list1}
Let $f\colon \R^m \to \R^m$ be a quasiregular mapping of transcendental type. Then: 
\begin{enumerate}
\item if $p\in\N$, $0<\eta<1$ and $R$ is sufficiently large, then 
\[ A_R(f) \subset A_R(f^p) \subset A_{\eta R}(f);\]
\item if $\R^m\backslash A_R(f)$ has a bounded component, then $A_R(f)$ and $A(f)$ are spider's webs;
\item if $G$ is a bounded component of $\R^m\backslash A_R^L(f)$, then $\partial G \subset A_R^L(f)$ and $f^n$ is a proper map of $G$ onto a bounded component of $\R^m\backslash A_R^{n+L}(f)$ for $n \in \N$;
\item if $\R^m\backslash A_R^L(f)$ has a bounded component, then $A_R^L(f)$ is a spider's web and hence every component of $\R^m\backslash A_R^L(f)$ is bounded;
\item $A_R(f)$ is a spider's web if and only if for each $L$, $A_R^L(f)$ is a spider's web;
\item if $R'>R>R_0$, then $A_R(f)$ is a spider's web if and only if $A_{R'}(f)$ is a spider's web.
\end{enumerate}
\end{proposition}
\begin{proof}
Part (i) is \cite[Theorem 2.6]{RS}, and the proof carries over
with~\cite[Lemma~2.4]{RS} replaced by Theorem~\ref{1x}.
Part (ii) is \cite[Theorem 1.4]{RS}, parts (iii)-(vi) are \cite[Lemma 7.1 (a)-(d)]{RS}.
\end{proof}
Note that (i) also gives $A(f) = A(f^p)$ for $p \in \N$, as was already
proved in Proposition~\ref{properties}, (ii).

We next define two sequences which in dimension $2$ are 
called the sequences of fundamental holes and loops for $A_R(f)$.

\begin{definition}
If $A_R(f)$ is a spider's web, then for $n \geq 0$ we denote by $H_n$ the component of $\R^m\backslash A_R^n(f)$ that contains $0$ and by $L_n$ the boundary of $H_n$.
\end{definition}

\begin{proposition}
\label{list2}
Let $f\colon \R^m \to \R^m$ be a $K$-quasiregular mapping of transcendental type. Then we have the following: 
\begin{enumerate}
\item for $n \geq 0$, we have $B(0,M^n(R,f)) \subset H_n$, $L_n \subset A_R^n(f)$ and $H_n \subset H_{n+1}$;
\item for $n \in \N$ and $k \geq 0$, we have $f^n(H_k) = H_{k+n}$ and $f^n(L_k) = L_{k+n}$;
\item there exists $N \in \N$ such that if $n \geq N$ and $k \geq 0$, we have $L_k \cap L_{k+n} = \emptyset$;
\item if $L \in \Z$ and $G$ is a component of $\R^m\backslash A_R^L(f)$, then for $n$ sufficiently large, we have $f^n(G) = H_{n+L}$ and $f^n(\partial G) = L_{n+L}$;
\item all components of $\R^m\backslash A(f)$ are compact if $A_R(f)$ is a spider's web;
\item $A_R(f^n)$ is a spider's web if and only if $A_R(f)$ is a spider's web.
\end{enumerate}
\end{proposition}

\begin{proof}
Parts (i)-(iv) are \cite[Lemma 7.2 (a)-(e)]{RS}, part (v) is \cite[Theorem 1.6]{RS} 
and part (vi) is \cite[Theorem 8.4]{RS}.
\end{proof}

The following characterization of the spider's web structure for $A_R(f)$ was proved for dimension $2$ in \cite[Theorem 8.1]{RS}. 

\begin{proposition}
\label{afchar}
Let $f\colon \R^m\to\R^m$ be a quasiregular mapping of transcendental type,
$R_0$ as in~\eqref{1a} and $R>R_0$.
Then $A_R(f)$ is a spider's web if and only if there exists a sequence
$(G_n)_{n=1}^{\infty}$ of bounded topologically convex domains such that, for all $n \in\N$,
\[ B(0,M^n(R,f)) \subset G_n,\]
and $G_{n+1}$ is contained in a bounded component of $\R^m \backslash f(\partial G_n)$.
\end{proposition}

This proposition has the following corollary, see \cite[Corollary 8.2]{RS}.

\begin{corollary}
\label{afchar2}
Let $f\colon \R^m\to\R^m$ be a quasiregular mapping of transcendental type,
$R_0$ as in~\eqref{1a}, $R>R_0$ and recall the minimum modulus function $m(r,f)$.
Then $A_R(f)$ is a spider's web if there exists a sequence $(\rho_n)_{n=1}^{\infty}$ such that
\[ \rho_n > M^n(R,f),\]
and
\[ m(\rho_n,f) \geq \rho_{n+1}.\]
\end{corollary}
\begin{proof}[Proof of Theorem~\ref{dsspider}]
By Proposition~\ref{list1} (vi), we may restrict to large values of~$R$.
Note then  that $M^n(R,f)$ is also large,
for all $n\in\N$.
Thus we may assume that for all $n\in\N$ there exists 
$\rho_n\in[2M^n(R,f),2\alpha M^n(R,f)]$ such that
\[
m(\rho_n,f)\geq \delta M(2M^n(R,f),f).
\]
By Lemma~\ref{MAr} we have 
\[
\delta M(2M^n(R,f),f)\geq 2\alpha M^{n+1}(R,f)\geq \rho_{n+1}
\]
for all $n\in\N$, provided $R$ is large enough. 
The conclusion now follows from Corollary~\ref{afchar2}.
\end{proof}

\section{A spider's web example}
\label{spiderexample}
In this section, we briefly outline the salient points of the class of quasiregular mappings $f\colon \R^m \to \R^m$ constructed by Drasin and Sastry and then show such mappings have a spider's web structure for $A(f)$. 

Drasin and Sastry \cite{DS} build quasiregular mappings of transcendental type with prescribed (slow) growth. This is achieved by starting with a positive continuous increasing function $\nu$ which is almost flat, that is,
that
\[ r\nu'(r) < \nu(r)/2, \:\:\:\:\: r\nu'(r) = o(\nu(r)),\]
as $r\to \infty$. Then
\begin{equation}
\label{meq} 
M(r,f) = \exp \int_1^r \frac{\nu(t)}{t} \: dt 
\end{equation}
for $r\geq 1$. We remark that \cite{DS} only states asymptotic equality in \eqref{meq}. However, it is not hard to see that equality is achieved for $x=(r,0,\ldots,0)$ where $r \geq 1$.

For $0<r<s$, define
\[ A_{\infty}(r,s) = \{ x \in \R^m \colon  r< \| x\|_{\infty} <s\}. \]
There exist $n_0 \in \N$ and a sequence of closed sets
\[ V_n = \overline{A_{\infty}}(r_n,s_n) \]
for $n \geq n_0$, so that the mapping $f$ restricted to the sets $V_n$ satisfies:
\begin{enumerate}[(i)]
\item $\nu(r_n) = n$, $\nu(s_n) \in [n+1/(m+1), n+1)$ and $s_n/r_n \to \infty$;
\item in $V_n$, $f$ coincides with a quasiregular power-type mapping of degree $j_n$ (see \cite[p.13]{Rickman} and \cite{DS}), where $(j_n)_{n=1}^{\infty}$ is an increasing sequence in $\N$;
\item if $\partial B(0,r) \subset V_n$, there exists $\delta>0$ independent of $n$ such that 
\[ m(r,f) \geq \delta M(r,f)\]
since $f$ behaves like a power mapping on $V_n$;
\item by \cite[(2.10)]{DS}, for $n\geq n_0$ we have
\[
\log \left ( \frac{r_{n+1}}{s_n} \right ) = C(m) \log \left ( \frac{n+1}{n} \right ),
\]
where $C(m)$ is a constant depending only on $m$. This condition says that the region where $f$ does not behave like a power mapping gets relatively smaller as $n$ increases;
\item for $n\geq n_0$ there exists a constant $\alpha=\alpha(m)>1$ depending only on $m$ such that if $\partial B(0,r)$ intersects $A_{\infty}(s_n,r_{n+1})$ then $\partial B(0,\alpha r)\subset V_{n+1}$.
\end{enumerate}

We remark that the delicate part of the construction of $f$ is to interpolate between the $V_n$ to increase the degree whilst keeping $f$ quasiregular. 

Let $E\subset (0,\infty)$ be defined as follows: $r\in E$ if and only if $\partial B(0,R)$ is contained in the region where $f$ behaves like a power mapping. By property (v) above, if we choose $n_0$ large enough, there exists $\alpha >1$ such that for $r\geq r_{n_0}$ there is a corresponding $s\in [r,\alpha r]$ with $s\in E$. Then by property (iii) and the fact that $M(r,f)$ is increasing in $r$,
\[ m(s,f) \geq \delta M(s,f) \geq \delta M(r,f).\]
Therefore $f$ satisfies the hypotheses of Theorem \ref{dsspider} and hence $A(f)$ is a spider's web.

\section{Failure of log-convexity of the maximum modulus}
\label{failureconvex}
In this section, we  prove Theorem \ref{notconvex}. 
As remarked in the introduction, this is a generalization of an idea of Dan Nicks, who considered compositions of quasiconformal radial mappings analogous to $h$ below, and certain transcendental entire functions in the plane. We will replace the transcendental entire function with a mapping which is similar to the type considered in the previous section, but now the growth function $\nu$ is allowed to be discontinuous. This still yields a quasiregular mapping, but allows greater flexibility in its properties.

Let $\rho \in (0,1)$. Fix $r_1>1$ and define $r_{n+1}=r_n^2$ for $n\in\N$. Define a radial quasiconformal map $h\colon \R^m\to \R^m$ by
\[
  h(x) = \left\{
  \begin{array}{l l}
    r_1x, & \quad |x|\in [0,r_1]\\
    |x|^{1/\rho }x r_n^{1-1/\rho}, & \quad |x| \in [r_n, r_n^{1+\rho }] \\
    r_n^2x, &\quad |x| \in [ r_n^{1+\rho}, r_n^2].\\
  \end{array} \right.
\]
Next define
\[
  \nu(r) = \left\{
  \begin{array}{l l}
    1, & \quad r\in [0,r_1]\\
    n, & \quad r \in [r_{n-1},r_{n}], \:\:\:\:\: n \geq 2. \\
  \end{array} \right.
\]
We define a function $f$ in a similar way to those constructed in \cite{DS}, but now using the intervals $[r_{n},r_{n+1}]$
(recall that $s_n \in (r_n,r_{n+1})$ and $f$ behaves like a power mapping on $\overline{A_{\infty}(r_n,s_n)}$) and subject to the growth condition
\[ M(r,f) = \exp \int_1^r \frac{\nu(t)}{t} \: dt.\]
We reiterate that this positive increasing function does not satisfy the conditions for $\nu$ considered in \cite{DS} since it is not continuous. 
However, using the same method as \cite[Lemma 3.7]{DS}, one can see that this mapping is indeed quasiregular.
One can calculate that if $r\in [r_{n-1},r_n]$, then
\begin{align*} \int_1^r \frac{\nu (t) }{t}\: dt &= n\log r  - \log r_{n-1} - \ldots - \log r_2 -\log r_1\\
&= n \log r - (2^n-1)\log r_1.
\end{align*}
Hence 
\[ \log M(r,f) = \psi (\log r),\]
where $\psi$ is a positive continuous piecewise linear function with
\[ \psi(t) = nt + d_n, \:\:\:\:\: t\in [\log r_{n-1}, \log r_n],\]
and
\[ d_n =   (1-2^n)\log r_1.\]

Define the quasiregular mapping $F\colon \R^m \to \R^m$ by $F=f \circ h$. Then if $r\in [r_n^{1+\rho}, r_n^2]$, the construction of $h$ and the fact that $r_n^2r \in [r_n^{3+\rho}, r_n^4] \subset [r_{n+1}, r_{n+2}]$ yield that
\begin{align*}
\log M(r,F) &= \log M(r, f \circ h) = \log M(r_n^2r,f) \\
&= \psi( \log r_n^2r) \\
&= (n+2) \log r + (n+2) \log r_n^2 + d_{n+2}.
\end{align*}
Hence for $r\in [r_n^{1+\rho}, r_n^2]$ we have 
\[ \frac{ \log M(r,F) }{\log r} = (n+2) + \frac{\psi (\log r_n^2)}{\log r}.\]
Since $\psi$ is a positive function,  $\log M(r,F) / \log r$ decreases on $[r_n^{1+\rho}, r_n^2]$. The union of these intervals has lower logarithmic density at least $(1-\rho) / (1+\rho)$, since $r_n = r_1^{2^{n-1}}$. Choosing $\rho$ close enough to zero implies the result.


\begin{thebibliography}{99}

\bibitem{Bell1967}
H. Bell,
On fixed point properties of plane continua,
{\it Trans. Amer. Math. Soc.}, {\bf 128} (1967), 539--548.

\bibitem{B}
W. Bergweiler, 
Fixed points of composite entire and quasiregular maps, 
{\it Ann. Acad. Sci. Fenn. Math.}, {\bf 31} (2006), 523--540.

\bibitem{B1}
W. Bergweiler, 
Karpinska's paradox in dimension 3, 
{\it Duke Math. J.}, {\bf 154} (2010), 599--630.

\bibitem{B2}
W. Bergweiler,
Iteration of quasiregular mappings, 
{\it Comput. Methods Funct. Theory}, {\bf 10} (2010), 455--481.

\bibitem{B2013}
W. Bergweiler,
Fatou-Julia theory for non-uniformly quasiregular maps.
{\em Ergodic Theory Dy\-nam.\ Systems},
{\bf 33} (2013), 1--23.

\bibitem{BFLM}
W. Bergweiler, A. Fletcher, J. K. Langley, J. Meyer,
The escaping set of a quasiregular mapping,
{\it Proc. Amer. Math. Soc.}, {\bf 137} (2009) 641--651.

\bibitem{BH}
W. Bergweiler, A. Hinkkanen, 
On semiconjugation of entire functions,
{\it Math. Proc. Cambridge Philos. Soc.}, {\bf 126} (1999), 565--574. 

\bibitem{BN}
W. Bergweiler, D. A. Nicks,
Foundations for an iteration theory of entire quasiregular maps, 
to appear in {\em Israel J. Math.}

\bibitem{BRS}
W.\ Bergweiler, P.\ J.\  Rippon and G.\ M.\  Stallard, Dynamics of
meromorphic functions with direct or logarithmic singularities,
{\it Proc.\ Lond. Math.\ Soc.},
 {\bf 97} (2008), 368--400.

\bibitem{Blokh2009}
A. Blokh,  L. Oversteegen, 
A fixed point theorem for branched covering  maps of the plane,
{\it Fund. Math.},
{\bf 206} (2009), 77--111.

\bibitem{DP}
D. Drasin, P. Pankka,
Sharpness of Rickman's Picard theorem in all dimensions,
arXiv: 1304.6998.

\bibitem{DS}
D. Drasin, S. Sastry, Periodic quasiregular mappings of finite order,
{\it Rev. Mat. Iberoamericana}, {\bf 19} (2003), 755--766.

\bibitem{E}
A. E. Eremenko, 
On the iteration of entire functions, 
{\it Dynamical systems and ergodic theory}, Banach Center Publications 23, Polish Scientific Publishers, Warsaw, 1989, 339--345.

\bibitem{FN} 
A. Fletcher, D. A. Nicks,
Quasiregular dynamics on the $n$-sphere,
{\it Ergodic Theory Dynam. Systems}, {\bf 31} (2011), 23--31.

\bibitem{Macintyre1938}
A.~J.~Macintyre,
Wiman's method and the `flat regions' of integral functions,
{\it Quart. J. Math., Oxford Ser.},
{\bf 9} (1938), 81--88.

\bibitem{MBP} H. Mihaljevic-Brandt, J. Peter, Poincar\'{e} functions with spiders' webs,
{\it Proc. Amer. Math. Soc.}, {\bf 140} (2012), 3193--3205. 

\bibitem{Miniowitz} R. Miniowitz, Normal families of quasimeromorphic mappings,
{\it Proc. Amer. Math. Soc.}, {\bf 84} (1982), 35--43.


\bibitem{Newman}
M. H. A. Newman, 
{\it Elements of the topology of plane sets of points}, Cambridge, 1961.

\bibitem{O}
J. W. Osborne,
The structure of spider's web fast escaping sets,
{\it Bull. Lond. Math. Soc.},
{\bf 44} (2012), 503--519. 

\bibitem{Rickman80}
S.\ Rickman,
On the number of omitted values of entire quasi\-regular mappings,
{\it J. Analyse Math.}, {\bf 37} (1980), 100--117.

\bibitem{Rickman85}
S.\ Rickman,
The analogue of Picard's theorem for quasi\-regular mappings in dimension three,
{\it Acta Math.},
{\bf 154} (1985), 195--242.

\bibitem{Rickman} S. Rickman, 
{\it Quasiregular mappings}, Ergebnisse der Mathematik und ihrer
Grenzgebiete 26, Springer, 1993.

\bibitem{RS1}
P. Rippon, G. Stallard,
On questions of Fatou and Eremenko,
{\it Proc. Amer. Math. Soc.}, {\bf 133} (2005), 1119--1126.

\bibitem{RS}
P. Rippon, G. Stallard, 
Fast escaping points of entire functions,
{\it Proc. Lond. Math. Soc.},
{\bf 105} (2012), 787--820.

\bibitem{RS2012}
P. Rippon, G. Stallard,
A sharp growth condition for a fast escaping spider's web,
{\it Adv. Math.},
{\bf 244} (2013), 337--353.

\bibitem{Siebert2006}
H. Siebert,
Fixed points and normal families of quasi\-regular mappings,
{\it J.\ Anal.\ Math.},  
{\bf 98}  (2006), 145--168.

\end{thebibliography}
\end{document}